\documentclass[journal]{IEEEtran}

\usepackage{cite}

\usepackage{graphicx}
\usepackage{pgf}
	
\usepackage[cmex10]{amsmath}
\usepackage{amsthm}
\usepackage{amsmath}
\usepackage{amssymb}
\usepackage{enumitem}
\usepackage{balance}
\usepackage{array}
\usepackage{bm}
\usepackage{color}

\newtheorem{thm}{Theorem}[section]

\newtheorem{lemma}[thm]{Lemma}

\theoremstyle{definition}
\newtheorem{defin}[thm]{Definition}

\theoremstyle{remark}
\newtheorem{remark}[thm]{Remark}

\theoremstyle{remark}

\def\R {{\Bbb R}}

\newcommand{\hsk}{\hskip 0.3cm}

\newcommand{\nl}{\newline}
\newcommand{\ben}{\begin{enumerate}}

\newcommand{\een}{\end{enumerate}}
\newcommand{\bit}{\begin{itemize}}

\newcommand{\eit}{\end{itemize}}

\def\eps{\varepsilon}

\def\dispsum{\displaystyle \sum}
\def\meansum{\frac{1}{N}\dispsum_{n=1}^{N}}

\def\QED{\nobreak\quad\ifmmode\roman{Q.E.D.}\else{\rm Q.E.D.}\fi}

\newcommand{\abs}[1]{\left\lvert#1\right\rvert}
\def\norm#1{\left\Vert#1\right\Vert}
\def\BracketRef#1{(\ref{#1})}
\def\mbf#1{\bm{#1}}

\begin{document}

\title{Finite sample performance of linear least squares estimation}


\author{Michael~Krikheli$^{1,2}$\thanks{$^{1}$The authors are with the Faculty of Engineering, Bar-Ilan University, 52900, Ramat-Gan. Corresponding author email: michael.krih@gmail.com.} and~Amir~Leshem$^{1,2}$
\thanks{$^2$ This work was supported by ISF grant 1644/18. Part of this work has been presented in ~\cite{krikheli2016finite} and \cite{krikheli2017finite}. }
}

\maketitle
\begin{abstract}
Linear Least Squares is a very well known technique for parameter estimation, which is used even when sub-optimal, because of its very low computational requirements and the fact that exact knowledge of the noise statistics is not required. Surprisingly, bounding the probability of large errors with finitely many samples has been left open, especially when dealing with correlated noise with unknown covariance. In this paper we analyze the finite sample performance of the linear least squares estimator. Using these bounds we obtain accurate bounds on the tail of the estimator's distribution. We show the fast exponential convergence of the number of samples required to ensure a given accuracy with high probability. We analyze a sub-Gaussian setting with a fixed or random design matrix of the linear least squares problem. We also extend the results to the case of a martingale difference noise sequence. Our analysis method is simple and uses simple $L_{\infty}$ type bounds on the estimation error. We also provide probabilistic finite sample bounds on the estimation error $L_2$ norm. The tightness of the bounds is tested through simulation. We demonstrate that our results are tighter than previously proposed bounds for $L_{\infty}$ norm of the error. The proposed bounds make it possible to predict the number of samples required for least squares estimation even when the least squares is sub-optimal and is used for computational simplicity.
\end{abstract}
\section{Introduction}
\subsection{Related Work}
Linear least squares estimation has numerous applications in many fields. For instance, it was used in soft-decision image interpolation applications in ~\cite{zhang2008image}. The interpolation for each image block was based on linear least squares estimation. An extension to the method was proposed in ~\cite{hung2012robust} where weighted linear least squares was used instead of regular least squares. The weights were distributed according to the geometric properties of the pixels of interest and the residuals. Another field that uses linear least squares is source localization using signal strength, as in ~\cite{so2011linear}. In that paper, weighted linear least squares was used to find the distance of the received signals given the strength of the signals received in the sensors and the sensors' locations. An improved algorithm was proposed using constraints on the two estimated variables. Weighted least squares estimators were also used in the field of diffusion MRI parameters estimation ~\cite{veraart2013weighted}. It was shown that the weighted linear least squares approach has significant advantages because of its simplicity and good results. A standard analysis of estimation problems calculates the Cramer-Rao bound (CRB) and implements the asymptotic normality of the estimator. This type of analysis is asymptotic by nature. For some applications, see for instance ~\cite{leshem1999direction} the direction of arrival problems were analyzed in terms of the CRB. In ~\cite{stoica1989music} the ML estimator and MUSIC algorithms were studied and the CRB was calculated.\nl
The noise model differs in many applications of least squares and other optimization methods. Rather than the Gaussian model, a Gaussian mixture is used in many applications. For instance, in ~\cite{djuric2002sequential} a Gaussian mixture model of time-varying autoregressive process was assumed and analyzed. The Gaussian mixture model was used to model noise in underwater communication systems in ~\cite{banerjee2014performance}. Wiener filters in Guassian mixture signal estimation were analyzed in ~\cite{tan2014wiener}. In ~\cite{bhatia2007non} a likelihood based algorithm for Gaussian mixture noise was devised and analyzed in the terms of the CRLB. In ~\cite{wang1999robust} a robust detection technique using Maximum-Likelihood estimation was proposed for an impulsive noise modelled as a Gaussian mixture.\nl
In this work we consider sub-Gaussian noise, which is quite a general and an interesting noise framework. The Gaussian mixture model for instance is sub-Gaussian and our results are valid for this model. In the case of Gaussian noise, least squares coincides with the maximum likelihood estimator. Still in  many cases of interest least squares estimation is used in non-Gaussian noise as well as for computational simplicity. Specifically the sub-Gaussian noise model is of special interest in many applications, for example ~\cite{abbasi2011improved, jamieson2014lil, ai2014one, reynolds1995robust}.\nl
The least squares problem has been extensively explored for many years both for linear and non-linear regression. The strong consistency of the linear least squares was proved in ~\cite{lai1985strong}. Asymptotic bounds for fixed size confidence bounds were stated for example in ~\cite{gleser1965asymptotic}. Non-linear least squares problems have been thoroughly studied as well. The result in ~\cite{rao1984exponential} shows the exponential rate of convergence for the Gaussian and sub-Gaussian case of non-linear least squares problems when the parameter space is $\R$. The authors derived the exponential rate of convergence without calculating the specific number of samples needed. In ~\cite{rao1984rate} the authors showed the rate of convergence for non-linear least squares estimates with dependent errors. Asymptotic performance for the non-linear least squares was analyzed in ~\cite{wu1981asymptotic}. There, the necessary and sufficient conditions for strong and weak consistency were given and the asymptotic normality was shown. Large deviation results for non-linear least squares estimations were given in ~\cite{sieders1987large} and ~\cite{shuhe1993large}. Ridge regularized least squares models were studied in ~\cite{caponnetto2007optimal} by means of the optimal convergence rate. \nl
One approach to performance analysis of the least squares estimator is the use of asymptotic normality. However, this is insufficient for large deviations since by the  Berry-Esseen theorem ~\cite{esseen1942liapounoff, berry1941accuracy,  shevtsova2010improvement} the normal approximation error is $O\left(\frac{1}{\sqrt{N}}\right)$. Other approaches were suggested in a Bayesian setting in ~\cite{routtenberg2012general}.\nl

In the past few years, the finite sample behavior of least squares problems has been studied in ~\cite{oliveira2016lower,hsu2014random,audibert2010robust,audibert2011robust, tsakonas2014convergence}. Some of these results also analyze ridge regularized least squares models. The sample complexity of linear regressors was analyzed in ~\cite{shamir2015sample}. In recent years there have been significant advances in understanding the statistical performance of regularized LS and other sparsity regularized techniques \cite{bickel2009simultaneous, van2014asymptotically, nickl2013confidence}. Non asymptotic performance of generalized linear model is given in ~\cite{van2007non}.\nl
In many cases of interest the noise model used is not i.i.d but instead a martingale difference sequence model. This noise model is quite general and is used in various fields. For example the first order ARCH models introduced in ~\cite{engle1982autoregressive} are popular in economic theory. Moreover, ~\cite{lai1982least} analyzed similar least squares models with applications in control theory. An important case of a martingale difference noise which is highly relevant in signal processing and communication applications is a bounded or Gaussian signal passing through an FIR channel. Non i.i.d noise models appear naturally in many signal processing applications, for example ~\cite{li2000airborne, noam2006asymptotic, gobet1981spectral}. The asymptotic properties of these models have been analyzed in various papers; for example ~\cite{lai1983asymptotic, nelson1980note, christopeit1980strong}. The results show the strong consistency of the least squares estimator under martingale difference noise and for autoregressive models. The least squares efficiency in an autoregressive noise model was studied in ~\cite{kramer1980finite}. However, finite sample results were not presented. \nl
\subsection{Contribution}
In this paper we provide a finite sample analysis of linear least squares problems.
The results significantly extend the results of ~\cite{oliveira2016lower,hsu2014random,audibert2010robust,audibert2011robust} in several ways. First we provide $L_{\infty}$ non-parametric bounds on the distribution of the error in linear least squares (i.e. we only assume bounds on the moment generating function of the noise rather than knowing the noise distribution). We then provide new bounds on the number of samples required for a given performance level and a given outage probability. Then we extend the results for the very general model of sub-Gaussian martingale difference sequence (MDS) noise. For the bounded MDS noise case we provide tighter results. This allows us to compute the performance of linear least squares under very general conditions. Since the linear least squares solution is computationally simple it is used in practice even when it is sub-optimal. The analysis of this paper allows the designer to understand the loss due to the computational complexity reduction without the need for massive simulations. The fact that we only need knowledge of a sub-Gaussianity parameter of the noise allows us to use these bounds when the noise distribution
is unknown. Our previous conference paper ~\cite{krikheli2016finite} was limited to the $L_{\infty}$ case and i.i.d noise. As explained above, the current results are much more general.
\section{problem formulation}\label{sec:problem_formulation}
Consider a linear model with additive noise
\begin{equation}\label{eq:x_definition}
\mbf{x} = \mbf{A}\mbf{\theta}_0 + \mbf{v}
\end{equation}
where $\mbf{x} \in \R^{N\times 1}$ is our output, $\mbf{A} \in \R^{N\times p}$ is a known matrix with bounded random elements, $\mbf{\theta}_0 \in \R^p$ is the estimated parameter and $\mbf{v} \in \R^{N\times 1}$ is a noise vector with independent and sub-Gaussian elements\footnote{For simplicity we only consider the real case. The complex case is similar with minor modifications.}. $N$ indicates the number of samples used in the model.\nl
Many real world noise models are sub-Gaussian; for instance, Gaussian, finite Gaussian mixture, all kinds of bounded variables, and any combination of the above. Many real-world models assume this kind of noise model.\nl
The linear least squares cost function is defined as:
\begin{equation}
J^N_{0}\left(\mbf{\theta},\mbf{x}\right) = \left(\mbf{x} - \mbf{A}\mbf{\theta}\right)^T\left(\mbf{x} - \mbf{A}\mbf{\theta}\right).
\end{equation}

Given $N$ samples the least squares solution is given by
\begin{equation}\label{def:theta_N}
\hat{\mbf{\theta}}^N_{0} = \left(\mbf{A}^T\mbf{A}\right)^{-1}\mbf{A}^T\mbf{x} = \left(\meansum \mbf{a}_n\mbf{a}_n^T\right)^{-1} \meansum \mbf{a}_n^Tx_n
\end{equation}
where $\mbf{a}_n^T, \hsk n=1 \dots N$ are the rows of $\mbf{A}$ and $x_n, \hsk n=\hsk 1 \dots N$ are the data samples. $\hat{\mbf{\theta}}^N_{0}$ is the value that minimizes the cost function $J^N_{0}$.
If the expected value of the noise is $0$ then the estimator is unbiased, i.e. the expected value of the estimator is the true parameter.\nl
We want to study the tail distribution of $\norm{\hat{\mbf{\theta}}^N_{0} - \mbf{\theta}_0}_{\infty}$ or more specifically to obtain bounds of the form
\begin{equation}\label{eq:wanted_inequality}
P\left(\norm{\hat{\mbf{\theta}}^N_{0} - \mbf{\theta}_0}_\infty > r\right) < \eps
\end{equation}
as a function of $N$. Furthermore, given $r$, $\eps$ we want to calculate the number of samples needed $N\left(r,\eps\right)$ to achieve the above inequality.
We analyze the case where $\mbf{A}$ is random with bounded elements. We also state theorems for the simpler case when $\mbf{A}$ is fixed. \nl
Throughout this paper we use the following mathematical notations:
\begin{defin} \ \\\label{definitions}
\begin{enumerate}
\item{Let $x$ be a random variable defined on the probability space $\left(\Omega,F, P\right)$ and denote $E\left(x\right)$ the expectation of $x$.}
\item{The logarithmic moment generating function of a random variable $x$ is: \nl
$\Lambda_{x}\left(s\right) := \log E[e^{sx}]$.}
\item{Let $\mbf{B} \in \R^{p\times p}$ be a square matrix; we define the operators $\lambda_{max}\left(\mbf{B}\right)$ and $\lambda_{min}\left(\mbf{B}\right)$ to give the maximal and minimal eigenvalues of $\mbf{B}$ respectively.}
\item{Let $\mbf{C}$ be a matrix. The spectral norm for matrices is given by \nl $\norm{\mbf{C}} \doteq \sqrt{\lambda_{max}\left(\mbf{C^T}\mbf{C}\right)}$.}
\item{A random variable $v$ with $E\left(v\right) = 0$ is called sub-Gaussian if its moment generating function exists and $E\left(\exp\left(sv\right)\right) \leq \exp\left(\frac{s^2R^2}{2}\right)$ ~\cite{vershynin2010introduction, eldar2012compressed}. The minimal $R$ that satisfies this inequality is called the sub-Gaussian parameter of the random variable $v$ and we say that $v$ is sub-Gaussian with parameter $R$.}\label{sub_Gaussian_def}
\item{A random variable $v$ with $E\left(v\right) = 0$ is called sub-exponential if its moment generating function exists and $\forall \abs{s} \leq \frac{c}{\lambda}$, $E\left(\exp\left(sv\right)\right) \leq \exp\left(Cs^2\lambda^2\right)$ where $c, C$ are absolute constants ~\cite{vershynin2010introduction, eldar2012compressed}.}
\item{We denote
$\tilde{\lambda}\left(\mbf{A}\right) \doteq \lambda_{max}\left(\left(\frac{1}{N}\mbf{A}^T\mbf{A}\right)^{-1}\right)$}
\end{enumerate}
\end{defin}
\begin{remark}
Since the sub-Gaussian property is crucial to our analysis we cite for completeness some well known properties of sub-Gaussian variables. See ~\cite{vershynin2010introduction} for a more detailed analysis
\begin{itemize}
\item{Let $x_1, \dots x_n$ be sub-Gaussian random variables with parameter $R$ then $\dispsum_{i=1}^n x_i$ is also sub-Gaussian.}
\item{Let $x$ be a sub-Gaussian random variable and $a > 0$ then, $ax$ is also sub-Gaussian.}
\item{Let $x_1, \dots x_n$ be sub-Gaussian random variables with parameter $R$ then for every $\mbf{a} \doteq \left(a_1, \dots, a_n\right)$ $P\left(\dispsum_{i=1}^na_ix_i > t\right) \leq \exp\left(-\frac{t^2}{2R^2\norm{\mbf{a}}^2}\right)$.}
\item{Let $x$ be a random variable. $x$ is a sub-Gaussian random variable if and only if $x^2 - E\left(x^2\right)$ is a sub-exponential random variable.}
\end{itemize}
\end{remark}
\begin{remark}
Using definition \ref{sub_Gaussian_def}, it is easy to see that a centered Gaussian random variable is also a sub-Gaussian random variable. Assume that $x \sim N\left(0,\sigma^2\right)$, then the moment generating function of $x$ is $M(s) = E\left(\exp\left(sx\right)\right) = \exp\left(\frac{s^2\sigma^2}{2}\right)$. Therefore, by definition \ref{sub_Gaussian_def} $x$ is also sub-Gaussian with parameter $\sigma$.
\end{remark}

\section {Main Result}\label{sec:main_result}
In this section we formulate and prove the main result of this paper.
\nl
This section starts with a statement of the assumptions used throughout the paper. We then state the main theorem and discuss it. We further explain the outline of the proof for of the main theorem and then provide a detailed proof. \nl
We start by making the following assumptions on our model:
We assume that the noise is zero mean and that $N > p$.
\begin{enumerate}[label=\bfseries A\arabic*:]
\item $P\left(\abs{a_{ni}} \leq \alpha\right) = 1 \hsk \forall n = 1 \dots N \hsk \forall i = 1 \dots p$
\item $\forall N > 0$ there exists $\mbf{M} \in \R^{p\times p}$ such that $\mbf{M} = \frac{1}{N} E\left(\mbf{A}^T\mbf{A}\right)$. We denote $\sigma_{max} \doteq \lambda_{max}\left(\mbf{M}\right)$ and $\sigma_{min} = \lambda_{min}\left(\mbf{M}\right)$.
\item $\forall 1 \leq n \leq N$, $v_n$ are statistically independent from each other and from $\mbf{A}$.
\item $\forall 1 \leq n \leq N$, $v_n$ are sub-Gaussian random variables with parameter $R$ (see definition \ref{sub_Gaussian_def}).
\end{enumerate}
The facts that the noise is zero mean and that $N > p$ ensure that the design is correct and that the least squares solution can be achieved. Assumptions A1 and A2 are mild and achievable by normalizing each row of the design matrix with the proper scaling of the sub-Gaussian parameter. Assumption A3 is an i.i.d. setup assumption and assumption A4 is the sub-Gaussian noise assumption. Note that the set of assumptions is valid for any type of zero mean sub-Gaussian noise model which is a very wide family of distributions. In case where the noise does not have a zero mean and the mean is known, subtracting the mean from the model equation \BracketRef{eq:x_definition} will make the results valid in this case as well.\nl

The main theorem provides bounds on the distance between the finite sample least squares estimator and the real parameter. The theorem provides the number of samples required so that the distance between the estimator and the real parameter will be at most $r$ with probability $1-\eps$.
\begin{thm} \emph{\textbf{(Main Theorem)}} \label{thm:main_theorem} \ \\
Let $\mbf{x}$ be defined as in \BracketRef{eq:x_definition} and assume assumptions A1-A4. Let $\hat{\mbf{\theta}}^N_{0}$ be the least squares estimator defined in \BracketRef{def:theta_N} and let $\mbf{\theta}_0$ be the true parameter. Furthermore, let $\eps > 0$ and $r > 0$ be given, then $\forall N>N\left(r,\eps\right) $
\begin{equation}
P\left(\norm{\hat{\mbf{\theta}}^N_{0}-\mbf{\theta}_{0}}_{\infty} > r\right) < \eps,
\end{equation}

Where
\begin{equation}\label{eq:N_general}
N\left(r,\eps\right) = \displaystyle \max\left\{N_1\left(r\right), N_2\left(r,\eps\right), N_3\left(r,\eps\right), N_{rand}\left(\eps\right)\right\}
\end{equation}
and

\begin{equation}\label{eq:N_1}
N_1\left(r\right) = \frac{4\alpha^2R^2}{\sigma_{min}^2r^2},
\end{equation}

\begin{multline}\label{eq:N_2}
N_2\left(r,\eps\right) = \displaystyle \inf_{0<s<\frac{1}{2\alpha^2R^2}}\left\{\frac{1}{\sigma_{min}^2r^2s}\left(8\beta\left(s\right) \right.\right. \\ \left.\left.+ 2\sigma_{min}r\sqrt{2s\log \frac{3p}{\eps}}\right)\right\},
\end{multline}

\begin{equation}\label{eq:N_3}
N_3\left(r,\eps\right) \leq \inf_{0<s<\frac{1}{\alpha^2R^2}}\left\{
\sqrt{\frac{\log \frac{3p}{\eps}}{\frac{\sigma_{min}^2r^2s}{8} - \gamma\left(s\right)}}\right\},
\end{equation}

\begin{equation}\label{eq:N_rand}
N_{rand}\left(\eps\right) = \frac{4}{3}\frac{\left(6\sigma_{max} + \sigma_{min}\right)\left(p\alpha^2 + \sigma_{max}\right)}{\sigma_{min}^2}\log\left(\frac{3p}{\eps}\right),
\end{equation}

\begin{equation}\label{def:beta}
\beta\left(s\right) \doteq \alpha^2R^2p + \frac{\alpha^4R^4s^2}{1-2R^2s},
\end{equation}

\begin{equation}\label{def:gamma}
\gamma\left(s\right) \doteq \frac{\alpha^4R^4s^2}{2} + \frac{\alpha^8R^8s^4}{4\left(1-s^2\alpha^4R^4\right)}.
\end{equation}


\end{thm}
\subsection{Discussion} \
The importance of this result is that it gives an easily calculated bound on the number of samples needed for linear least squares problems. It shows a sharp convergence in probability as a function of $N$, and shows that the number of samples required to achieve an error less than $r$ with probability $1-\eps$ is $O\left(\frac{1}{r^2}\log \frac{1}{\eps}\right)$ and provides exact constants which allow easy computation of the bounds and not just the finite sample decay rate. \nl
The results in this work are given with an $L^{\infty}$ norm. The $L^{\infty}$ results can give confidence bounds for every coordinate of the parameter vector $\mbf{\theta}_0$. Results for other norms can be achieved as well using relationships between norms.\nl

\subsection{Proof Outline} \
In this subsection we present the outline of the proof for the main theorem. The proof is divided into two main parts. The first part of the proof provides bounds on the differences between the estimated parameter and the true parameter at each coordinate. It uses the sub-Gaussian assumption to provide bounds that are easy to calculate. The second part of the proof exploit the union bound to provide bounds on the $L_{\infty}$ norm of the error vector. Together the two parts finish the proof. We now discuss the main stages of the proof.\nl
We are interested in studying the term
\begin{equation}
P\left(\norm{\hat{\mbf{\theta}}^N_0 - \mbf{\theta}_0}_{\infty}>r\right).
\end{equation}
We start by showing that this is equivalent to analyzing the term
\begin{equation}\label{eq:l_infinity_matrix_probability}
P\left(\norm{\left(\mbf{A}^T\mbf{A}\right)^{-1}\mbf{A}^T\mbf{v}}_{\infty} > r\right).
\end{equation}
In order to study the infinity norm we analyze each of the elements of the error vector separately. We later use the union bound on the probability of error in each element. For each $1 \leq i \leq p$ we study the term
\begin{equation}
P\left(\abs{\left(\frac{1}{N}\mbf{A}^T\mbf{A}\right)^{-1}\meansum a_{ni}v_n} > r\right).
\end{equation}
We continue by showing that this is bounded by
\begin{equation}
P\left(\abs{\meansum a_{ni}v_n} > \frac{r}{\lambda_{max}\left(\left(\frac{1}{N}\mbf{A}^T\mbf{A}\right)^{-1}\right)}\right).
\end{equation}
In order to analyze this term we start with showing that if $N > N_{rand}\left(\eps\right)$ the probability of the event
\begin{equation}\label{eq:e_rand}
E_{rand} \doteq \left\{\mbf{A} : \tilde{\lambda}\left(\mbf{A}\right) > \frac{2}{\sigma_{min}}\right\}
\end{equation}
is less than $\frac{\eps}{3}$. In other words $E_{rand}$ is defined as the event that the minimal eigenvalue of the sample covariance or the maximum eigenvalue of the inverse of the sample covariance is greater than $\frac{2}{\sigma_{min}}$. We show that the probability of this event is less than $\frac{\eps}{3}$ for large enough $N$. i.e., $\forall N>N_{rand}\left(\eps\right)$
\begin{equation}
P\left(\mbf{x} \in E_{rand}\right) \leq \frac{\eps}{3}.
\end{equation}
Assuming that $\mbf{A} \notin E_{rand}$ we use a Chernoff type bound to bound the number of samples needed to ensure that $\abs{\meansum a_{ni}v_n}$ is small. The bound $N > N_1\left(r\right)$ ensures that the Chernoff bound converges. The terms $N_2\left(r,\eps\right)$ and $N_3\left(r,\eps\right)$ ensure that for each element $1 \leq i \leq p$ the probability that $\abs{\meansum a_{ni}v_n}$ is larger than $\frac{r\sigma_{min}}{2}$ is bounded by $\frac{\eps}{3p}$. We use $N_2\left(r,\eps\right)$ to bound the probability of the event
\begin{equation}
E_2\left(i\right) \doteq \left\{\mbf{x} : \frac{1}{N^2}\dispsum_{n=1}^N a_{ni}^2v_n^2 > \frac{\sigma_{min}^2r^2}{8}\right\}
\end{equation}
given that $\mbf{A} \notin E_{rand}$. We use $N_3\left(r,\eps\right)$ to bound the probability of the event
\begin{equation}
E_3\left(i\right) \doteq \left\{\mbf{x} : \frac{1}{N^2}\dispsum_{n=1}^N\dispsum_{l=1,l\neq n}^N a_{ni}a_{li}v_nv_l > \frac{\sigma_{min}^2r^2}{8}\right\}
\end{equation}
given that $\mbf{x} \notin E_{rand}$. In order to achieve these bounds we use the sub-Gaussian noise assumption to bound the moment generating functions of some of the terms. This technique enables us to achieve bounds that are easy to calculate.\nl
The next part of the proof uses the union bound on the elements of the error vector. Using the bounds so far, we can use the union bound to bound the probability that any of the  elements of the error vector is too large. By doing so, we achieve a bound on the probability of three events, where each event has a probability bounded by $\frac{\eps}{3}$. The events are defined as $E_{rand}$ as defined in \BracketRef{eq:e_rand} and $E_2$ and $E_3$ where
\begin{multline}
E_2 \cap E_{rand}^C \doteq \\
\left\{\mbf{x} : \exists 1 \leq i \leq p \text{ such that } \mbf{x} \in E_2\left(i\right) \wedge \mbf{x} \notin E_{rand} \right\}
\end{multline}
and
\begin{multline}
E_3 \cap E_{rand}^C \doteq \\
\left\{\mbf{x} : \exists 1 \leq i \leq p \text{ such that } \mbf{x} \in E_3\left(i\right) \wedge \mbf{x} \notin E_{rand}\right\}.
\end{multline}
We show that $\forall N > N\left(r,\eps\right)$, $P\left(\mbf{x} \in E_3 \wedge \mbf{x} \notin E_{rand}\right) \leq \frac{\eps}{3}$ ,$P\left(\mbf{x} \in E_2 \wedge \mbf{x} \notin E_{rand}\right) \leq \frac{\eps}{3}$ and $P\left(\mbf{x} \in E_{rand}\right) \leq \frac{\eps}{3}$.
We prove that if none of these events occurs, the performance of the least squares estimator is as required. In order to finish the proof we need to bound the probability that any of these events will occur. We use the union bound again to achieve that $\forall N > N\left(r,\eps\right)$
\begin{multline*}
P\left(E_{rand} \cup E_2 \cup E_3\right) \leq \\ P\left(E_{rand}\right) + P\left(E_2 \cap E_{rand}^C\right) + P\left(E_3 \cap E_{rand}^C\right) \leq \\ \frac{\eps}{3} + \frac{\eps}{3} + \frac{\eps}{3} = \eps
\end{multline*}
i.e., the probability that \BracketRef{eq:l_infinity_matrix_probability} is violated is less than the sum of the probabilities of the events $E_2 \cap E_{rand}^C,E_3 \cap E_{rand}^C,E_{rand}$.\nl
The next sections provide the details of the proof and some extensions.
\subsection{Proof of the main theorem}
In this subsection we prove the main theorem of this paper. The result is for a general noise model. We start the proof by stating two auxiliary lemmas. The proofs of these lemmas are given in the appendix.
\begin{lemma}\label{lemma:general_least_squares_calculation}
Let $\mbf{x}$ be defined as in \BracketRef{eq:x_definition}. Assume A1-A3 hold. Furthermore, let $\hat{\mbf{\theta}}^N_{0}$ be defined in \BracketRef{def:theta_N} and let $r > 0$ be given, then
\begin{equation}
\begin{array}{lcl}
&P\left(\abs{\left(\hat{\mbf{\theta}}^N_0 - \mbf{\theta}_0\right)_i} > r\right)&\\
\leq& P\left(\abs{\meansum a_{ni}v_n} > \frac{r}{\tilde{\lambda}\left(\mbf{A}\right)}\right)&
\end{array}
\end{equation}
\end{lemma}
\begin{proof}
The proof of this lemma is given in appendix A.
\end{proof}

\begin{lemma}\label{lemma:omega_bound}
Under assumptions A1,A2 and $\forall N \geq N_{rand}\left(\eps'\right)$
\begin{equation}
P\left(\tilde{\lambda}\left(\mbf{A}\right) \geq \frac{2}{\sigma_{min}}\right) \leq \eps',
\end{equation}
where
\begin{equation}
N_{rand}\left(\eps'\right) = \frac{4}{3}\frac{\left(6\sigma_{max} + \sigma_{min}\right)\left(p\alpha^2 + \sigma_{max}\right)}{\sigma_{min}^2}\log\left(\frac{p}{\eps'}\right).
\end{equation}
\end{lemma}
\begin{proof}
The proof for this lemma is given in appendix B.
\end{proof}
We are now ready to prove the main theorem.
\begin{proof}
We want to study the term
\begin{equation} \label{eq:main_probability}
P\left(\abs{\left(\hat{\mbf{\theta}}^N_0 - \mbf{\theta}_0\right)_i} > r\right).
\end{equation}

Using lemma \ref{lemma:general_least_squares_calculation} it suffices to bound the term
\begin{equation}
P\left(\abs{\meansum a_{ni}v_n} > \frac{r}{\tilde{\lambda}\left(\mbf{A}\right)}\right).
\end{equation}
Substituting we obtain
\begin{equation}
\begin{array}{lcl}
& P\left(\abs{\left(\hat{\mbf{\theta}}_{0}^N - \mbf{\theta}_{0}\right)_i > r}\right) \\
\leq& P\left(\abs{\meansum a_{ni}v_n} > \frac{r}{\tilde{\lambda}\left(\mbf{A}\right)}\right) &\\
\leq &P\left(\left(\meansum a_{ni}v_n\right)^2 > \frac{r^2}{\tilde{\lambda}^2\left(\mbf{A}\right)}\right). &
\end{array}
\end{equation}

The next stages of the proof use a union bound type argument. We first exclude a set of possible values of $\mbf{A}$; $E_{rand} = \left\{\mbf{x} : \tilde{\lambda}\left(\mbf{A}\right) \geq \frac{2}{\sigma_{min}}\right\}$. We show that $\forall N \geq N_{rand}\left(\eps\right)$, $P\left(\mbf{x} \in E_{rand}\right) \leq \frac{\eps}{3}$. Then, we will show that if $N > \max\left\{N_1\left(r\right), N_2\left(r,\eps\right)\right\}$ and given the set $E_2\left(i\right) = \left\{\mbf{x} : \frac{1}{N^2}\dispsum_{n=1}^Na_{ni}^2v_n^2 > \frac{\sigma_{min}^2r^2}{8} \right\}$, $P\left(\mbf{x} \in E_2\left(i\right) \wedge \mbf{x} \notin E_{rand}\right) \leq \frac{\eps}{3p}$. We next show that $\forall N > \max\left\{N_1\left(r\right), N_3\left(r,\eps\right)\right\}$ and given the set $E_3\left(i\right) = \left\{\mbf{x} : \frac{1}{N^2}\dispsum_{n=1}^N\dispsum_{l=1, l\neq n}^N a_{ni}a_{li}v_nv_l > \frac{\sigma_{min}^2r^2}{8}\right\}$, $P\left(\mbf{x} \in E_3\left(i\right) \wedge \mbf{x} \notin E_{rand}\right) \leq \frac{\eps}{3p}$. The next step is to use union bound over the elements of the error vector and use union bound on all the errors once more. \nl

We start by using lemma \ref{lemma:omega_bound} with the parameter $\eps' = \frac{\eps}{3}$
to ensure that $\forall N > N_{rand}\left(\eps\right)$, $P\left(\mbf{x} \in E_{rand}\right) = P\left(\tilde{\lambda}\left(\mbf{A}\right) \geq \frac{2}{\sigma_{min}}\right) \leq \frac{\eps}{3}$.
Therefore, for every $N > N_{rand}\left(\eps\right)$ with probability $1-\frac{\eps}{3}$, $\mbf{x} \notin E_{rand}$ and as a result
\begin{multline}\label{eq:random matrix lemma usage}
P\left(\abs{\left(\hat{\mbf{\theta}}_{0}^N - \mbf{\theta}_{0}\right)_i > r}\right) \\
\leq P\left(\left(\meansum a_{ni}v_n\right)^2 > \frac{\sigma_{min}^2r^2}{4}\right).
\end{multline}
From now on we assume that $\mbf{x} \notin E_{rand}$.

We want to use the Chernoff bound ~\cite{dubhashi2009concentration} to bound the probability \BracketRef{eq:main_probability}. In order to use the Chernoff bound we need
\begin{equation}\label{eq:expectation_inequality}
E\left(\left(\meansum a_{ni}v_n\right)^2\right) < \frac{\sigma_{min}^2r^2}{4}.
\end{equation}

Developing the term $\left(\meansum a_{ni}v_n\right)^2$ yields

\begin{equation} \label{eq:J_N_norm}
\begin{array}{lcl}
\left(\meansum a_{ni}v_n\right)^2 &=& \\
\frac{1}{N^2}\dispsum_{n=1}^N a_{ni}^2v_n^2 &+& \frac{1}{N^2}\dispsum_{n=1}^N\dispsum_{l=1, l\neq n}^N a_{ni}v_na_{li}v_l.
\end{array}
\end{equation}

Calculating the expectation term for this gives
\begin{equation}
\begin{array}{lcl}
E\left(\left(\meansum a_{ni}v_n\right)^2\right) = \frac{1}{N^2}\dispsum_{n=1}^NE\left(a_{ni}^2v_n^2\right) &\leq& \\
\frac{1}{N^2}\dispsum_{n=1}^NE\left(\alpha^2v_n^2\right) = \frac{\alpha^2}{N}E\left(v^2\right),&&
\end{array}
\end{equation}
where the first inequality follows from the fact that the noise is zero mean and that by the choice of $a_{ni}$, $\abs{a_{ni}} \leq \alpha$.
Using this, we can write the inequality ~\eqref{eq:expectation_inequality} as
\begin{equation}
\frac{\sigma_{min}^2r^2}{4} > \frac{\alpha^2}{N}E\left(v^2\right).
\end{equation}
Solving this gives the condition on $N$
\begin{equation}
N> \frac{4\alpha^2E\left(v^2\right)}{\sigma_{min}^2r^2}
\end{equation}
Using the fact that $v$ is sub-Gaussian and the properties of sub-Gaussian variables we know that:
\begin{equation}
E\left(v^2\right) \leq R^2.
\end{equation}
Substituting into the previous equation we obtain
\begin{equation}
N > \frac{4\alpha^2R^2}{\sigma_{min}^2r^2} \doteq N_1\left(r\right).
\end{equation}
We now turn to evaluate the two terms of equation ~\eqref{eq:J_N_norm} using Chernoff-like bounds.\nl
We denote by $E_2\left(i\right)$ the set
\begin{equation}\label{def:psi_2}
E_2\left(i\right) = \left\{\mbf{x} : \frac{1}{N^2}\dispsum_{n=1}^{N}a_{ni}^2v_n^2 > \frac{\sigma_{min}^2r^2}{8}\right\}.
\end{equation}
We study $P\left(\mbf{x} \in E_2\left(i\right) \wedge \mbf{x} \notin E_{rand}\right)$.
\begin{equation}
\begin{array}{lcl}
& P\left(\mbf{x} \in E_2\left(i\right)\wedge \mbf{x} \notin E_{rand}\right) & \\
=& P\left(\frac{1}{N^2}\dispsum_{n=1}^{N}a_{ni}^2v_n^2 > \frac{\sigma_{min}^2r^2}{8}\right) &\\
\leq& \exp\left(-\frac{\sigma_{min}^2r^2N^2s}{8}\right)E\left(\exp \left(s\dispsum_{n=1}^{N} a_{ni}^2v_n^2\right)\right) &\\
\leq& \exp\left(-\frac{\sigma_{min}^2r^2N^2s}{8}\right)E\left(\exp \left(s\dispsum_{n=1}^{N} \alpha^2v_n^2\right)\right) &\\
=& \exp\left(-\frac{\sigma_{min}^2r^2N^2s}{8}\right)E\left(\exp \left(s\alpha^2v^2\right)\right)^N.&
\end{array}
\end{equation}
The second inequality follows from the fact that $\abs{a_{ni}}\leq \alpha$.\nl

By the monotonicity of the log function, this is equivalent to
\begin{equation}
\begin{array}{lcl}
\log\left(P\left(\frac{1}{N^2}\dispsum_{n=1}^Na_{ni}^2v_n^2 > \frac{\sigma_{min}^2r^2}{8}\right) \right)&\leq& \\
-\frac{\sigma_{min}^2r^2N^2s}{8} + N\log E\left(\exp\left(s\alpha^2v^2\right)\right).&&
\end{array}
\end{equation}

To ensure that $\forall \eps > 0$
\begin{equation}
P\left(\frac{1}{N^2}\dispsum_{n=1}^{N}a_{ni}^2v_n^2 > \frac{\sigma_{min}^2r^2}{8}\right) < \frac{\eps}{3p},
\end{equation}
It is sufficient to solve the inequality
\begin{equation}
\log P\left(\frac{1}{N^2}\dispsum_{n=1}^{N}a_{ni}^2v_n^2 > \frac{\sigma_{min}^2r^2}{8}\right) <  \log \frac{\eps}{3p}.
\end{equation}

Evaluating this yields a quadratic inequality
\begin{equation}
-\frac{\sigma_{min}^2r^2sN^2}{8} + N\log E\left(\exp\left(s\alpha^2v^2\right)\right) < \log \frac{\eps}{3p}.
\end{equation}

Solving for $N$ gives the inequality
\begin{multline}
N > \frac{4}{\sigma_{min}^2r^2s}
\left(\log E\left(\exp\left(s\alpha^2v^2\right)\right) + \right.\\\left.\sqrt{\log ^2 E\left(\exp\left(s\alpha^2v^2\right)\right) + \frac{r^2s}{2}\log \frac{3p}{\eps}}\right) \leq
\\ \frac{1}{\sigma_{min}^2r^2s}\left(8\log E\left(\exp\left(s\alpha^2v^2\right)\right) \right. \\ \left.+ 2\sigma_{min}r\sqrt{2s\log \frac{3p}{\eps}}\right).
\end{multline}
We now use the sub-Gaussian assumption to bound the term $E\left(\exp\left(s\alpha^2v^2\right)\right)$. In order to achieve this we use remark 2.3 in ~\cite{hsu2012tail} to achieve
\begin{equation}
E\left(e^{s\alpha^2v^2}\right) \leq \exp\left(\alpha^2R^2s + \frac{\alpha^4R^4s^2}{1-2\alpha^2R^2s}\right) = \exp\left(\beta\left(s\right)\right).
\end{equation}
This is true $ \forall 0 < s <\frac{1}{2\alpha^2R^2}$.\nl
Substituting this bound into the previous equation and taking the infimum over $s$ gives the bound on $N$
\begin{multline}
 N > \displaystyle \inf_{0<s<\frac{1}{2\alpha^2R^2}}\left\{\frac{1}{\sigma_{min}^2r^2s}\left(8\beta\left(s\right) \right.\right. \\ \left.\left.+ 2\sigma_{min}r\sqrt{2s\log \frac{3p}{\eps}}\right)\right\} \doteq N_2\left(r,\eps\right).
\end{multline}
Therefore, as required choosing $N > N_2\left(r,\eps\right)$ ensures that $\forall 1 \leq i \leq p$
\begin{equation}
P\left(\mbf{x} \in E_2\left(i\right) \wedge \mbf{x} \notin E_{rand}\right) \leq \frac{\eps}{3p}.
\end{equation}

We now turn to evaluating the second term of ~\eqref{eq:J_N_norm}

We define
\begin{equation}\label{def:psi_3}
E_3\left(i\right) = \left\{\mbf{x}: \frac{1}{N^2}\dispsum_{n=1}^{N}\dispsum_{l=1, l\neq
n}^{N}a_{ni}a_{li}v_n v_l > \frac{\sigma_{min}^2r^2}{8}\right\}.
\end{equation}

\begin{equation}\label{eq:J_multiplication_bound}
\begin{array}{lcl}
&&P\left(\mbf{x} \in E_3\left(i\right) \wedge \mbf{x} \notin E_{rand}\right)\\
&\leq& \exp\left(-\frac{N^2\sigma_{min}^2r^2s}{8}\right)\times \\
&& E\left(\exp\left(s\dispsum_{n=1}^{N}\dispsum_{l=1,l\neq n}^{N}\alpha^2v_nv_l\right)\right)\\
 &=& \exp\left(-\frac{N^2\sigma_{min}^2r^2s}{8}\right)\times \\
&& \displaystyle \prod_{n=1}^{N}\prod_{l=1,l\neq n}^{N}E\left(\exp\left(s\alpha^2v_nv_l\right)\right) \\
&=& \exp\left(-\frac{N^2\sigma_{min}^2r^2s}{8}\right) \times \\
&&E\left(\exp\left(s\alpha^2v_1v_2\right)\right)^{N\left(N-1\right)},\\
\end{array}
\end{equation}
where $v_1$ and $v_2$ are two i.i.d random variables distributed according to the law of $v$.
By the monotonicity of the log function we achieve
\begin{multline}
\log P\left(\mbf{x} \in E_3\left(i\right) \wedge \mbf{x} \notin E_{rand}\right) <\\  -\frac{\sigma_{min}^2r^2s}{8}N^2 + N\left(N-1\right)\Lambda_{\mbf{v}_1\mbf{v}_2}\left(s\alpha^2\right).
\end{multline}
\nl
We are interested in solving the equation
\begin{equation}
\log P\left(\mbf{x} \in E_3\left(i\right) \wedge \mbf{x} \notin E_{rand}\right)< \log \frac{\eps}{3p}.
\end{equation}
Assigning the inequality derived in \eqref{eq:J_multiplication_bound} we obtain the quadratic inequality
\begin{equation}
-\frac{\sigma_{min}^2r^2s}{8}N^2 + N\left(N-1\right)\Lambda_{v_1v_2}\left(s\alpha^2\right) < \log \frac{\eps}{3p}.
\end{equation}
\nl
Solving this inequality yields
\begin{multline} \label{eq:second_bound_J_norm}
N > \frac{-\Lambda_{\mbf{v}_1\mbf{v}_2}\left(s\alpha^2\right)}{2\left(\frac{\sigma_{min}^2r^2s}{8} - \Lambda_{\mbf{v}_1\mbf{v}_2}\left(s\alpha^2\right)\right)} + \\  \frac{\sqrt{\Lambda_{\mbf{v}_1\mbf{v}_2}\left(s\alpha^2\right)^2 + 4\left(\frac{\sigma_{min}^2r^2s}{8} - \Lambda_{\mbf{v}_1\mbf{v}_2}\left(s\alpha^2\right)\right)\log \frac{3p}{\eps}}}{2\left(\frac{\sigma_{min}^2r^2s}{8} - \Lambda_{\mbf{v}_1\mbf{v}_2}\left(s\alpha^2\right)\right)} \\
\leq \sqrt{\frac{\log \frac{3p}{\eps}}{\left(\frac{\sigma_{min}^2r^2s}{8} - \Lambda_{\mbf{v}_1\mbf{v}_2}\left(s\alpha^2\right)\right)}}.
\end{multline}

We now use smoothing and the fact that $v_1$ and $v_2$ are sub-Gaussian random variables with parameter $R$. Using bounds on sub-Gaussian variables again ~\cite{hsu2012tail} we achieve the following bound
\begin{multline}\label{eq:gamma_bound}
\Lambda_{v_1v_2}\left(s\alpha^2\right) \leq \log E\left(\exp\left(s\alpha^2v_1v_2\right)\right) =\\
\log E_{v_2}E_{v_1|v_2}\left(e^{s\alpha^2v_1v_2}\right) \leq  \log E_{v_2}\left(\exp\left(\frac{s^2\alpha^4v_2^2R^2}{2}\right)\right)\\
\leq \log \exp\left(\frac{\alpha^4R^4s^2}{2} + \frac{\alpha^8R^8s^4}{4\left(1-\alpha^4s^2R^4\right)}\right)\\
= \frac{\alpha^4R^4s^2}{2} + \frac{\alpha^8R^8s^4}{4\left(1-\alpha^4s^2R^4\right)}
\doteq \gamma\left(s\right).
\end{multline}
This is true $\forall 0 < s < \frac{1}{\alpha^2R^2}$.

Substituting \BracketRef{eq:gamma_bound} into \BracketRef{eq:second_bound_J_norm} and taking infimum over the valid values of $s$ we obtain the bound
\begin{equation}
N \leq \inf_{0<s<\frac{1}{\alpha^2R^2}}\left\{
\sqrt{\frac{\log \frac{3p}{\eps}}{\frac{\sigma_{min}^2r^2s}{8} - \gamma\left(s\right)}}\right\} \doteq N_3\left(r,\eps\right).
\end{equation}

Choosing $N > N_3\left(r,\eps\right)$ will ensure $P\left(\mbf{x} \in E_3\left(i\right) \wedge \mbf{x} \notin E_{rand}\right) \leq \frac{\eps}{3p}$ $\forall 1 \leq i \leq p$.

Define
\begin{multline*}
E_2 \cap E_{rand}^C \doteq \\
\left\{\mbf{x} : \exists 1 \leq i \leq p \text{ such that } \mbf{x} \in E_2\left(i\right) \wedge \mbf{x} \notin E_{rand}\right\}.
\end{multline*}

and
\begin{multline*}
E_3 \cap E_{rand}^C \doteq \\
\left\{\mbf{x} : \exists 1 \leq i \leq p \text{ such that } \mbf{x} \in E_3\left(i\right) \wedge \mbf{x} \notin E_{rand}\right\}.
\end{multline*}
where $E_2\left(i\right)$ and $E_3\left(i\right)$ are defined in \BracketRef{def:psi_2} and \BracketRef{def:psi_3} respectively. Choosing $N > \max\left\{ N_{rand}\left(\eps\right),N_1\left(r\right), N_2\left(r,\eps\right), N_3\left(r,\eps\right)\right\} \doteq N\left(r,\eps\right)$, we obtain that $ \forall 1 \leq i \leq p$,
\begin{equation*}
P\left(\mbf{x} \in E_2\left(i\right) \wedge \mbf{x} \notin E_{rand}\right) \leq \frac{\eps}{3p}
\end{equation*} and
\begin{equation*}
P\left(\mbf{x} \in E_3\left(i\right) \wedge \mbf{x} \notin E_{rand} \right) \leq \frac{\eps}{3p}.
\end{equation*}
Using the union bound over the $p$ elements of the $\mbf{\theta}$ we obtain that
\begin{multline*}
P\left(\mbf{x} \in E_2 \wedge \mbf{x} \notin E_{rand}\right) \leq \\
\dispsum_{i=1}^pP\left(\mbf{x} \in E_2\left(i\right) \wedge \mbf{x} \notin E_{rand}\right) \leq \frac{\eps}{3}
\end{multline*}
and
\begin{multline*}
P\left(\mbf{x} \in E_3 \wedge \mbf{x} \notin E_{rand}\right) \leq \\
\dispsum_{i=1}^pP\left(\mbf{x} \in E_3\left(i\right) \wedge \mbf{x} \notin E_{rand}\right) \leq \frac{\eps}{3}.
\end{multline*}
We now use the union bound again to obtain that
\begin{multline*}
P\left(E_{rand} \cup E_2 \cup E_3\right) \leq \\ P\left(E_{rand}\right) + P\left(E_2 \cap E_{rand}^C\right) + P\left(E_3 \cap E_{rand}^C\right) \leq \\ \frac{\eps}{3} + \frac{\eps}{3} + \frac{\eps}{3} = \eps.
\end{multline*}
We thus show that $ \forall N > N\left(r,\eps\right)$ the probability that there is no $1 \leq i \leq p$
such that
\begin{equation}
P\left(\abs{\left(\hat{\mbf{\theta}^N_{0}} - \mbf{\theta}_{0}\right)_i} > r \right) > \eps
\end{equation}
and therefore,
\begin{equation}
P\left(\norm{\left(\hat{\mbf{\theta}^N_{0}} - \mbf{\theta}_{0}\right)}_{\infty} > r \right) \leq \eps.
\end{equation}
This completes the proof of the main theorem.

\end{proof}
\begin{remark}
Whereas the main theorem bounds are given for $L^{\infty}$ norm this does not limit the scope of the theorem. Bounds for other norms can be given using the relationships between norms. For instance, bounding the $L_2$ norm gives
\begin{equation}
P\left(\norm{\hat{\mbf{\theta}}_0^N - \mbf{\theta}_0}_2 > r\right) \leq
P\left(\norm{\hat{\mbf{\theta}}_0^N - \mbf{\theta_0}}_{\infty} > \frac{r}{\sqrt{p}}\right)
\end{equation}
Assigning $r' = \frac{r}{\sqrt{p}}$ to the main theorem provides a bound for $L_2$ norm. Similarly, other norms can be considered as well.
\end{remark}
\subsection{Tighter bounds that involve optimization}
In the proof of theorem ~\ref{thm:main_theorem} we used a few constants that can be optimized for better performance although calculating the bounds is harder. \nl
We define a parameter $0 < \tau < 1$. In order to find $N$ such that
\begin{equation}
P\left(\frac{1}{N^2}\dispsum_{n=1}^N a_{ni}^2v_n^2 + \frac{1}{N^2}\dispsum_{n=1}^N\dispsum_{l=1, l\neq n}^N a_{ni}v_na_{li}v_l > \frac{\sigma_{min}^2r^2}{4}\right),
\end{equation}
we split the equation into two using $\tau$ for parametrization; i.e. we find $N$ such that
\begin{equation}
P\left(\frac{1}{N^2}\dispsum_{n=1}^N a_{ni}^2v_n^2  > \tau\frac{\sigma_{min}^2r^2}{4}\right)
\end{equation}
and
\begin{equation}
P\left(\frac{1}{N^2}\dispsum_{n=1}^N\dispsum_{l=1, l\neq n}^N a_{ni}v_na_{li}v_l > \left(1-\tau\right)\frac{\sigma_{min}^2r^2}{4}\right).
\end{equation}
Using the same methods as in the proof of the main theorem we get
\begin{multline}
N_2\left(r,\eps,\tau\right) = \displaystyle \inf_{0<s<\frac{1}{2\alpha^2R^2}}\left\{\frac{1}{\tau\sigma_{min}^2r^2s}\left(4\beta\left(s\right) \right.\right. \\ \left.\left.+ \sigma_{min}r\sqrt{2s\log \frac{2}{\eps}}\right)\right\},
\end{multline}

and
\begin{equation}
N_3\left(r,\eps, \tau\right) = \inf_{0<s<\frac{1}{\alpha^2R^2}}\left\{
\sqrt{\frac{\log \frac{2}{\eps}}{\frac{\left(1-\tau\right)\sigma_{min}^2r^2s}{4} - \gamma\left(s\right)}}\right\},
\end{equation}

where $\beta\left(s\right)$ and $\gamma\left(s\right)$ are defined in \BracketRef{def:beta} and \BracketRef{def:gamma} respectively.
$N_1\left(r,\eps\right), N_{rand}\left(\eps\right)$ are left unchanged and defined in \BracketRef{eq:N_1}, \BracketRef{eq:N_rand} respectively. Using these values we have
\begin{multline}
N = \\
\displaystyle \min_{0 < \tau < 1}\max\left\{N_1\left(r,\eps\right), N_2\left(r,\eps,\tau\right), N_3\left(r,\eps,\tau\right), N_{rand}\left(\eps\right)\right\}.
\end{multline}
The proof of the theorem remains the same.

\subsection{Bounds on $\eps\left(r,N\right)$}
The bounds given in this paper are bounds on the number of required samples as a function of $\eps$ and $r$. Another useful question is the following: given $r$ and $N$ what is the probability that the distance between the estimator and the real parameter is at most $r$? In other words, we want to find an expression for $\eps$ as a function of $N$ and $r$. We discuss these expressions for the sub-Gaussian linear model below.
\begin{lemma}
Let $N$ and $r$ be given such that
\begin{equation}
N > \frac{4\alpha^2R^2}{\sigma_{min}^2r^2},
\end{equation}
then
\begin{equation}
P\left(\norm{\hat{\mbf{\theta}}^N_{0}-\mbf{\theta}_{0}}_{\infty}>r\right) < \eps,
\end{equation}
where \begin{equation}
\eps = \max\left\{\eps_2, \eps_3, \eps_{rand}\right\}
\end{equation}
and

\begin{multline}
\eps_2\left(r,N\right) \leq \\ \inf_{0 < s < \frac{1}{2\alpha^2R^2}}
\left\{3p\exp\left(-\frac{\left(\sigma_{min}^2r^2sN - 8\beta\left(s\right)\right)^2}{8s\sigma_{min}^2r^2}\right)\right\},
\end{multline}

where $\beta\left(s\right)$ is defined in \BracketRef{def:beta}.

\begin{multline}
\eps_{3}\left(r,N\right) \leq \\ \inf_{0 < s < \frac{1}{\alpha^2R^2}}\left\{3p\exp\left(-N^2\left(\frac{\sigma_{min}^2r^2s}{8} - \gamma\left(s\right)\right)\right)\right\},
\end{multline}
where $\gamma\left(s\right)$ is defined in \BracketRef{def:gamma}.
\begin{equation}
\eps_{rand} =
3p\exp\left(-\frac{3}{4}\frac{N\sigma_{min}^2}{\left(6\sigma_{max} + \sigma_{min}\right)\left(p\alpha^2 + \sigma_{max}\right)}\right).
\end{equation}

\end{lemma}
\begin{proof}
The proof is a straightforward calculation given the bounds on $N_1$, $N_2$, $N_3$ and $N_{rand}$ in ~\BracketRef{eq:N_1}, ~\BracketRef{eq:N_2} and ~\BracketRef{eq:N_3} ~\BracketRef{eq:N_rand} respectively.
\end{proof}

\subsection{Least squares for the bounded noise case}
In some cases of interest the noise sequence is bounded almost surely; see for example ~\cite{ rangan2001recursive, walter1990estimation, stojanovic2003modeling}. In these cases we can obtain simpler and tighter bounds to analyze the least squares method. The proof of the following is similar to the proof of theorem \ref{thm:main_theorem} and is deferred to appendix C.
\begin{thm}\label{thm:bounded_noise_theorem}
Let $\mbf{x}$ be defined as in ~\BracketRef{eq:x_definition}. Moreover, assume assumptions A1-A3 and assume moreover that
\begin{enumerate}[label=\bfseries A\arabic*:]
\setcounter{enumi}{3}
\item $E\left(v_n^2\right) = R \hsk \forall n$.
\item $P\left(v_n \leq b\right) = 1 \hsk \forall n$.
\end{enumerate}
Then, there exists $N\left(r,\eps\right)$ such that $\forall \eps,r > 0$ and $N > N\left(r,\eps\right)$
\begin{equation}
P\left(\norm{\hat{\mbf{\theta}}^N_0 - \mbf{\theta}_0}_{\infty} > r\right) \leq \epsilon
\end{equation}
where $N\left(r,\eps\right) = \max \left\{N_{rand}\left(\eps\right), N_1\left(r, \eps\right)\right\}$
and
\begin{equation}
N_{rand}\left(\eps\right) = \frac{4}{3}\frac{\left(6\sigma_{max} + \sigma_{min}\right)\left(p\alpha^2 + \sigma_{max}\right)}{\sigma_{min}^2}\log\left(\frac{3p}{\eps}\right),
\end{equation}
\begin{equation}
N_1\left(r,\eps\right) = \frac{2\alpha^2b^2}{r^2\sigma_{min}^2}\log\left(\frac{3p}{\eps}\right).
\end{equation}
\end{thm}
\begin{proof}
The proof is given in Appendix C.
\end{proof}

\section{Martingale difference noise sequences}\label{sec:martingale_difference}
In many cases of interest the noise is correlated and the analysis above is not accurate. For example in ~\cite{engle1982autoregressive, li2000airborne, noam2006asymptotic, gobet1981spectral}. In this section we examine the least squares problem under the assumption that the noise is a sub-Gaussian martingale difference sequence; i.e., we change assumptions A4 and A3 to the following assumptions:
\begin{enumerate}[label=\bfseries A\arabic*:]
\setcounter{enumi}{5}
\item $E\left(v_n | F_{n-1}\right) = 0$. Where $F_{n-1}$ is a filtration, $v_n$ are independent of $\mbf{A}.$
\item $P\left(\abs{v_n} \leq b\right) = 1$.
\item The martingale difference sequence is $R$ sub-Gaussian. i.e. $E\left(e^{s v_n} | F_{n-1}\right) \leq e^{\frac{s^2 R^2}{2}}$.
\end{enumerate}
We now state the theorems for the martingale difference noise sequence version

\begin{thm}{\emph{\textbf{(sub-Gaussian martingale differences)}}}\label{thm:subgaussian_martingale_difference} \ \\
Let $\mbf{x}$ be defined as in ~\BracketRef{eq:x_definition} and assume assumptions A1,A2 and assumptions A6,A8. Let $\eps > 0$ and $r > 0$ be given and $\hat{\mbf{\theta}}_0^N$ and $\mbf{\theta}_0$ be defined as previously, then $\forall N > N\left(r,\eps\right)$
\begin{equation}
P\left(\norm{\hat{\mbf{\theta}}_0^N - \mbf{\theta}_0}_{\infty} > r\right) < \eps
\end{equation}
where
\begin{equation}
N\left(r,\eps\right) = \max \left\{N_1\left(r,\eps\right), N_{rand}\left(\eps\right)\right\},
\end{equation}

\begin{equation}
N_1\left(r,\eps\right) = \frac{8\alpha^2R^2}{r^2\sigma_{min}^2}\log \frac{2p}{\eps}
\end{equation}
and
\begin{equation}
N_{rand}\left(\eps\right) = \frac{4}{3}\frac{\left(6\sigma_{max} + \sigma_{min}\right)\left(p\alpha^2 + \sigma_{max}\right)}{\sigma_{min}^2}\log \frac{2p}{\eps}.
\end{equation}

\end{thm}
\begin{proof}
The first part of the proof remains the same as in the bounded martingale differences case; i.e., for every
$N > N_{rand}\left(r,\eps\right)$ and given the definition of $E_{rand}$ in \BracketRef{eq:e_rand}
\begin{equation}\label{subgaussian_martingale_matrix_bound}
P\left(\mbf{x} \in E_{rand}\right) = P\left(\tilde{\lambda}\left(\mbf{A}\right) \geq \frac{2}{\sigma_{min}}\right) \leq \frac{\eps}{2}.
\end{equation}
We now want to find $N_1\left(r,\eps\right)$ such that $\forall N > N_1\left(r,\eps\right)$
\begin{multline}
P\left(\mbf{x} \in E_2\left(i\right) \wedge \mbf{x} \in E_{rand}\right) = \\
P\left(\frac{1}{N}\dispsum_{n=1}^N a_{ni}v_n > \frac{r\sigma_{min}}{2}\right) \leq \frac{\eps}{2p}
\end{multline}
where $E_2\left(i\right)$ is defined in \BracketRef{def:psi_2}.
We cannot use Azuma's inequality since unlike to the previous theorem the martingale difference sequence is no longer bounded. However, we can prove a concentration result for sub-Gaussian martingale difference sequence using similar methods to ~\cite{thoppe2015concentration}. We start by bounding $E\left(\exp\left(s\dispsum_{n=1}^N a_{ni}v_n\right)\right)$ and then we use Markov's inequality.
\begin{equation}
\begin{array}{lcl}
&E\left(\exp\left(s\dispsum_{n=1}^N a_{ni}v_n\right)\right)& \\
\leq & E\left(\exp\left(s\dispsum_{n=1}^N \alpha v_n\right)\right) & \\
= & E\left(\exp\left(s \alpha\dispsum_{n=1}^{N-1}v_n\right)\right) E\left(\exp\left(s\alpha v_N\right) | F_{N-1}\right) & \\
\leq &E\left(\exp\left(s\alpha \dispsum_{n=1}^{N-1}v_n\right)\right) \exp\left(\frac{s^2\alpha^2R^2}{2}\right).&
\end{array}
\end{equation}
The first inequality follows from assumption A1 and the fact that the exponent is a monotonically increasing function. The second inequality follows from assumption A8. Iterating this procedure yields
\begin{equation}\label{sub_Gaussian_expectation_bound}
E\left(\exp\left(s\dispsum_{n=1}^N a_{ni}v_n\right)\right) \leq \exp\left(\frac{Ns^2\alpha^2R^2}{2}\right).
\end{equation}
Looking now at the original equation we have
\begin{equation}
\begin{array}{lcl}
&P\left(\mbf{x} \in E_2\left(i\right) \wedge \mbf{x} \notin E_{rand}\right)&\\
=&P\left(\dispsum_{n=1}^N a_{ni}v_n > \frac{Nr\sigma_{min}}{2}\right)&\\
= &P\left(\exp\left(s\dispsum_{n=1}^N a_{ni}v_n \right)> \exp \left(s\frac{Nr\sigma_{min}}{2}\right)\right)& \\
\leq & E\left(\exp\left(s\dispsum_{n=1}^N a_{ni}v_n \right)\right) \exp \left(-s\frac{Nr\sigma_{min}}{2}\right) & \\
\leq &\exp\left(\frac{Ns^2\alpha^2R^2}{2}\right)\exp \left(-s\frac{Nr\sigma_{min}}{2}\right)& \\
 = & \exp\left(\frac{N}{2}\left(s^2\alpha^2R^2 - sr\sigma_{min}\right)\right)
\end{array}
\end{equation}
The first inequality follows from Markov's inequality. The second inequality follows from equation ~\BracketRef{sub_Gaussian_expectation_bound}. The inequality is true $\forall s > 0$. Optimizing over $s$ yields $s = \frac{r\sigma_{min}}{2\alpha^2R^2}$. Assigning to the last equation we obtain
\begin{equation}
P\left(\dispsum_{n=1}^N a_{ni}v_n > \frac{Nr\sigma_{min}}{2}\right) < \exp\left(-\frac{Nr^2\sigma_{min}^2}{8\alpha^2R^2}\right).
\end{equation}
Choosing $N > \frac{8\alpha^2R^2}{r^2\sigma_{min}^2}\log \frac{2p}{\eps}$ ensures that
\begin{multline}
P\left(\mbf{x} \in E_2\left(i\right) \wedge \mbf{x} \notin E_{rand}\right) = \\
P\left(\meansum a_{ni}v_n > \frac{r\sigma_{min}}{2}\right) < \frac{\eps}{2p}.
\end{multline}
Using the same union bound strategy as in the previous theorem finishes the proof.
\end{proof}
\begin{thm}{\emph{\textbf{(bounded martingale differences)}}}\label{thm:martingale_difference} \ \\
Let $\mbf{x}$ be defined as in ~\BracketRef{eq:x_definition} and assume assumptions A1,A2 and assumptions A6,A7. Let $\eps > 0$ and $r > 0$ be given and $\hat{\mbf{\theta}}_0^N$ and $\mbf{\theta}_0$ be defined as previously, then $\forall N > N\left(r,\eps\right)$
\begin{equation}
P\left(\norm{\hat{\mbf{\theta}}_0^N - \mbf{\theta}_0}_{\infty} > r\right) < \eps
\end{equation}
where
\begin{multline}
N\left(r,\eps\right) = \max\left\{\frac{8\alpha^2b^2}{r^2\sigma_{min}^2}\log\frac{2p}{\eps},\right. \\\left. \frac{4}{3}\frac{\left(6\sigma_{max} + \sigma_{min}\right)\left(p\alpha^2 + \sigma_{max}\right)}{\sigma_{min}^2}\log \frac{2p}{\eps}\right\}.
\end{multline}
\end{thm}
\begin{proof}
This proof is similar to the bounded noise case with the change for martingale difference noise.
\end{proof}
\begin{remark}
While the theorems in this paper deal with the case that the design matrix $\mbf{A}$ is random, they can easily be applied to the case that $\mbf{A}$ is known. In this case the eigenvalues of $\frac{1}{N}\mbf{A}^T\mbf{A}$ can be calculated and therefore, there is no need of using lemma ~\ref{lemma:omega_bound}. If we denote $\sigma_{min} \doteq \lambda_{min}\left(\frac{1}{N}\mbf{A}^T\mbf{A}\right)$ then, we need to remove $N_{rand}$ from the calculation of $N\left(r,\eps\right)$, also $\alpha = \displaystyle \max_{ij}a_{ij}$ is easily calculated. The rest of the proof remains the same. We show this with the example of sub-Gaussian martingale difference sequence model with fixed design matrix. As we said the same technique can be applied to all theorems.
\end{remark}
\begin{thm}{\emph{\textbf{(sub-Gaussian martingale differences with fixed design matrix)}}}\label{thm:subgaussian_martingale_difference_fixed_matrix} \ \\
Let $\mbf{x}$ be defined as in ~\BracketRef{eq:x_definition} and assume assumptions A6,A8. Moreover, assume that $\mbf{A}$ is known, $\sigma_{min} = \lambda_{min}\left(\frac{1}{N}\mbf{A}^T\mbf{A}\right)$ and $\alpha = \displaystyle\max_{i,j}a_{ij}$. Let $\eps > 0$ and $r > 0$ be given and $\hat{\mbf{\theta}}_0^N$ and $\mbf{\theta}_0$ be defined as previously, then $\forall N > N\left(r,\eps\right)$
\begin{equation}
P\left(\norm{\hat{\mbf{\theta}}_0^N - \mbf{\theta}_0}_{\infty} > r\right) < \eps
\end{equation}
where
\begin{equation}
N\left(r,\eps\right) = \frac{8\alpha^2R^2}{r^2\sigma_{min}^2}\log \frac{2p}{\eps}
\end{equation}
\end{thm}
\begin{figure}
    \centering
    \includegraphics[width=0.4\paperwidth]{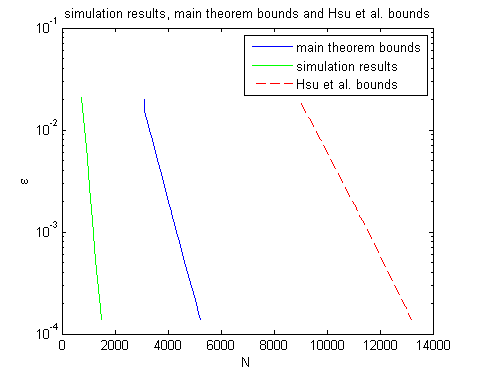}
    \caption{Simulation results, main theorem bounds and Hsu et al. bounds for Gaussian mixture noise with $R = 0.1$, $r = 0.01$, $\sigma_{min} = \sigma_{max} = 1$ and  $p = 8$.}
\end{figure}

\begin{figure}
    \centering
    \includegraphics[width=0.4\paperwidth]{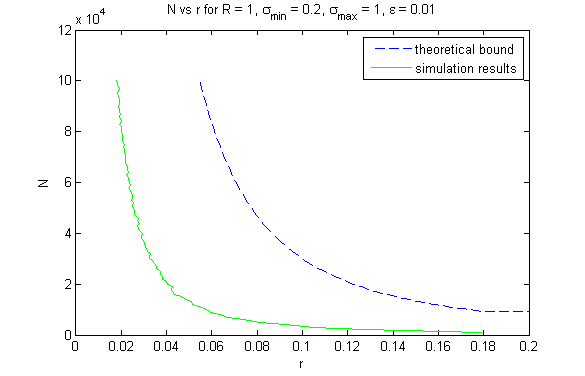}
    \caption{Simulation results and theoretical bounds for uniform noise with $\eps = 0.01, R = 1,  \sigma_{min} = 0.2, \sigma_{max} = 1$ and $p = 2$. The graph is for $N$ as a function of $r$.}
\end{figure}

\begin{figure}
    \centering
    \includegraphics[width=0.4\paperwidth]{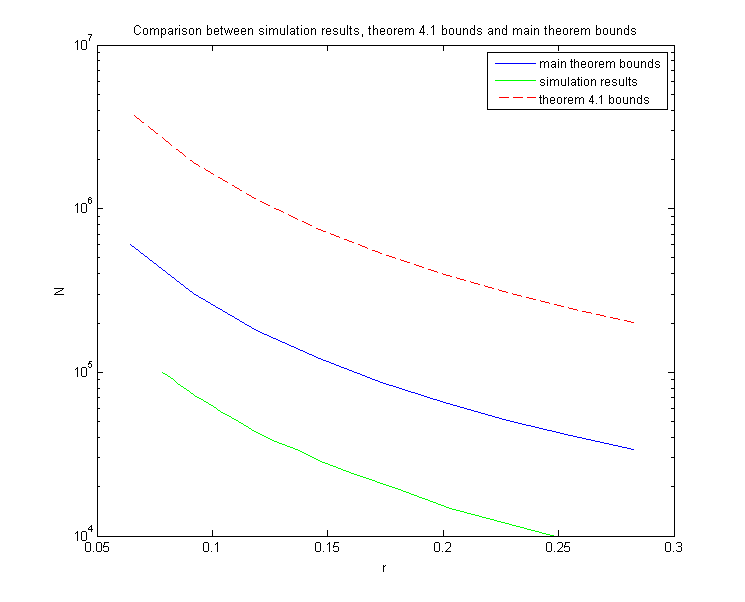}
    \caption{Simulation results, theorem ~\ref{thm:subgaussian_martingale_difference} bounds and theorem ~\ref{thm:main_theorem} with $\eps = 0.05, R = 10, \sigma_{min} = 1, \sigma_{max} = 1$ and $p = 4$. The graph is for $N$ as a function of $r$.}
\end{figure}

\begin{figure}
    \centering
    \includegraphics[width=0.4\paperwidth]{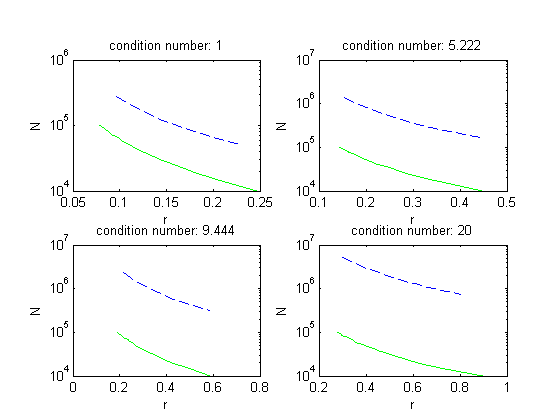}
    \caption{Simulation results and theoretical bounds for multiple condition numbers and $\eps = 0.05, R = 10, \sigma_{max} = 1$ and $p = 4$. The graphs are for $N$ as a function of $r$.}
\end{figure}


\begin{figure}
    \centering
    \includegraphics[width=0.4\paperwidth]{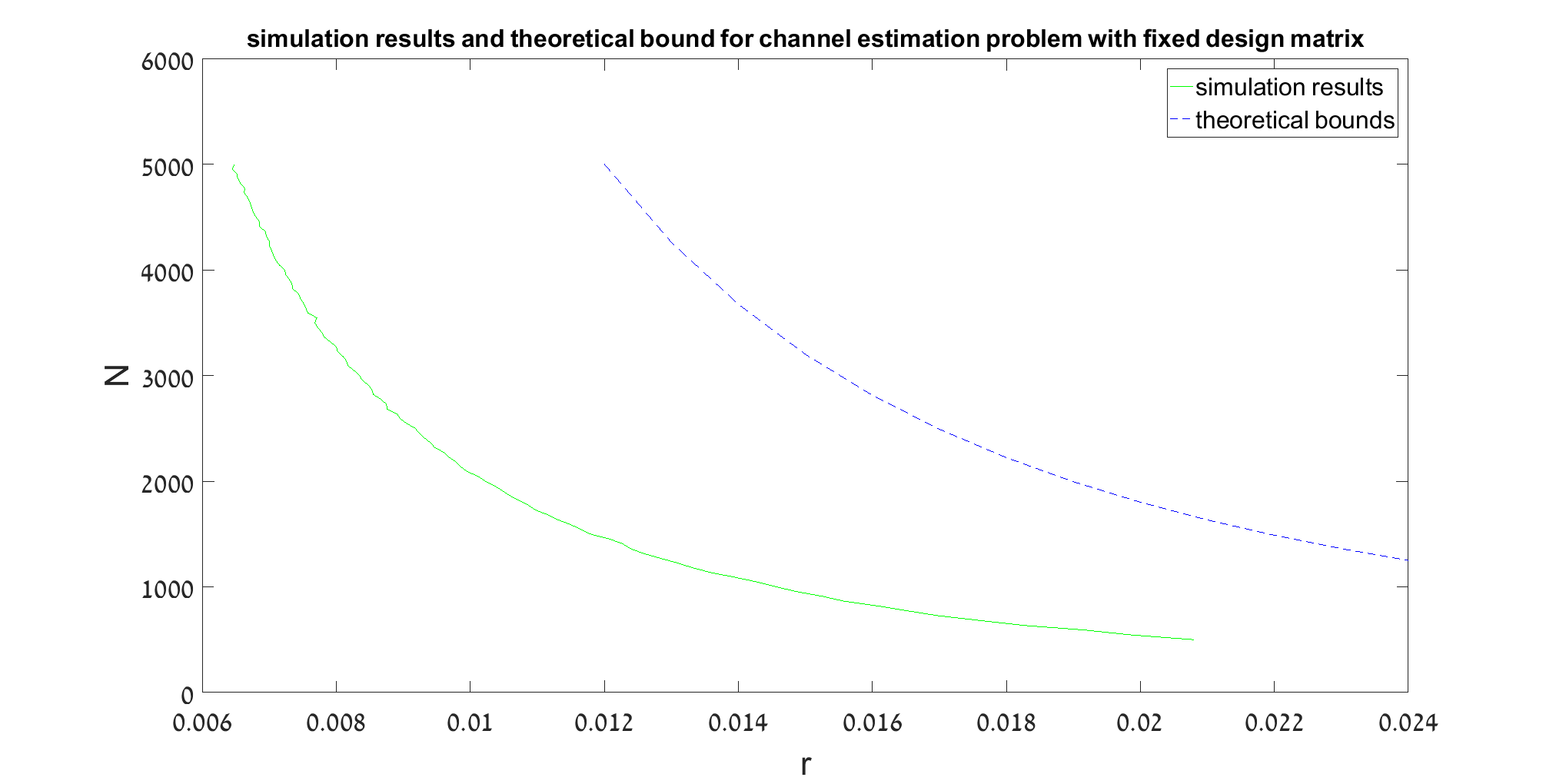}
    \caption{Simulation results and martingale difference theorem bounds for channel estimation problem with fixed design matrix and i.i.d noise with $\eps = 0.01$, $R = 0.1$ and $p = 8$. The graph is for $N$ as a function of $r$.}
\end{figure}

\begin{figure}
    \centering
    \includegraphics[width=0.4\paperwidth]{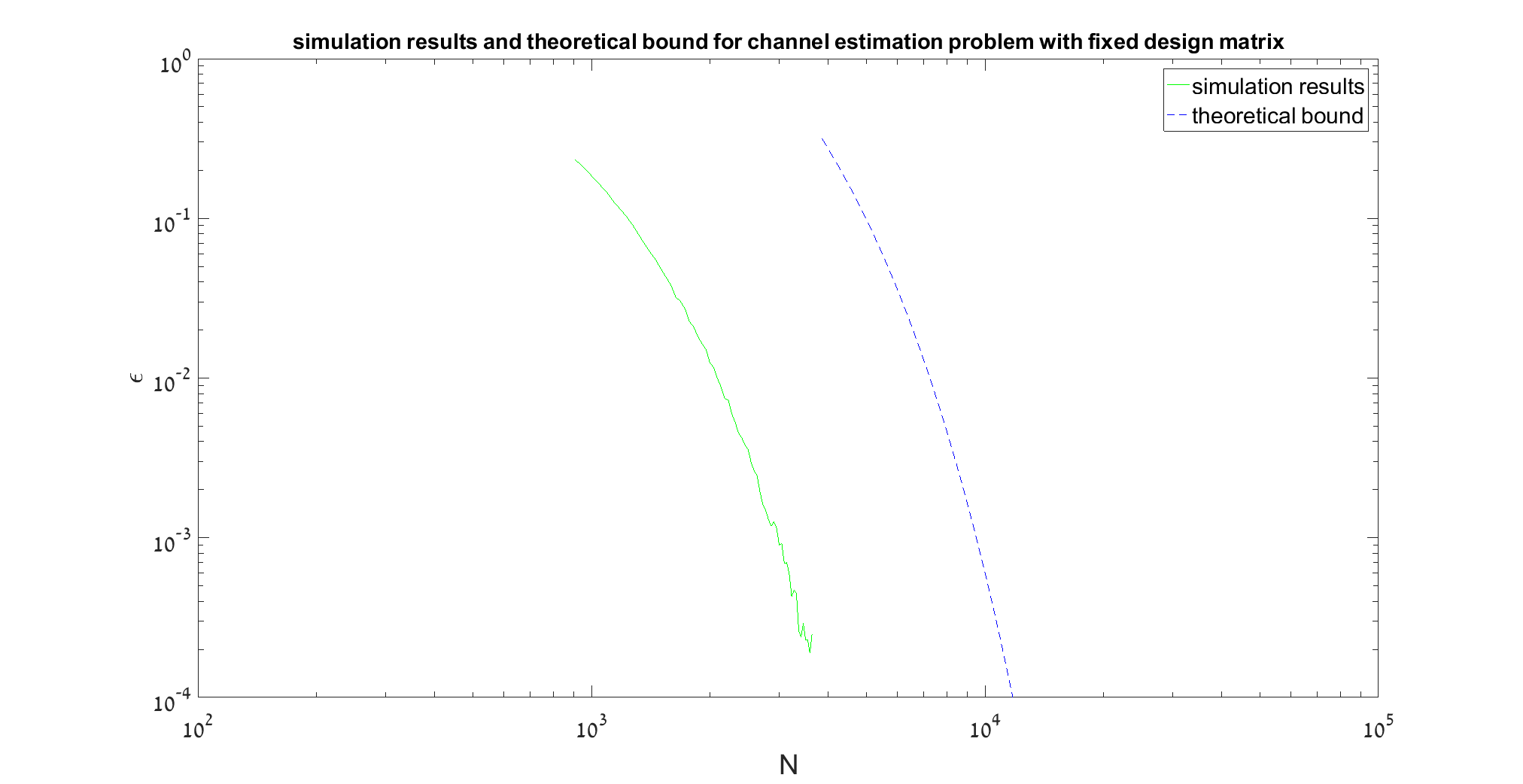}
    \caption{Simulation results and martingale difference theorem bounds for channel estimation problem with fixed design matrix and i.i.d noise with $r = 0.01$, $R = 0.1$ and $p = 8$. The graph is for $\eps$ as a function of $N$.}
\end{figure}

\section{simulation results}\label{sec:simulation_results}
In this section we describe simulation results that demonstrate the bounds. Our model matrix in the first two simulations was a random matrix with bounded elements and variance 1 in each dimension. We tested the bound for two interesting sub-Gaussian noise settings. The first was uniform plus Gaussian noise. The second setting was a Gaussian mixture with two Gaussians, one with high variance and small probability (0.1) and the other with low variance and high probability (0.9). The simulations were conducted using 50,000 samples for each value of $N$. The simulation results shown in Fig 1. indicate that the overall performance of the bound is similar to the performance in the simulations. We also compared our bounds to the bounds generated in ~\cite{hsu2014random}. We clearly see that the bounds in the main theorem in this paper are tighter. Fig. 2. shows the behavior of our bounds as a function of $r$. We set $\eps = 0.01$ and calculated the required number of samples to achieve a deviation $r$. The bounds and the simulation results exhibit similar properties. Fig. 3 depicts the performance of the main theorem bound compared to the martingale difference bound. The figure shows that although the bounds are close to one another, the main theorem bound is tighter. This is due to the finer analysis of the model in the main theorem proof as well as the fact that the model used is better. Fig 4. depicts the bounds and the simulation results for different condition numbers. We see that the bound performs better when the condition number is small. These results suggest that future work can reduce the bound on the number of required samples by a factor of the condition number.
\nl

We now provide a simulation of a problem of high importance in signal processing, i.e., channel estimation in the presence of interfering signal. To that end, we utilize the case where the design matrix is known, and generated by the training signal.
Assume that a training sequence of length $N$, $s_0,...s_{N-1}$ is transmitted through an unknown channel. The noise plus interference is modelled as an $R$ sub-Gaussian martingale difference composed of another transmitter in the area as well as receiver noise. As we now show the external interference which is a bounded communication signal passing through a linear time invariant channel is indeed a martingale difference sequence. The fixed design matrix is given by
\begin{equation}\label{eq:channel_estimation_fixed_a_definition}
	\mbf{A} =
	\begin{bmatrix}
   		s_{0}       & 0 & 0 & \dots & 0 \\
    		s_{1}       & s_{0} & 0 & \dots & 0 \\
    		\hdotsfor{5} \\
    		s_{p-1}       & s_{p-2} & s_{p-3} & \dots & s_{0} \\
		s_{p}    		& s_{p-1} & s_{p-2} & \dots & s_{1} \\
		s_{p+1}    		& s_{p} & s_{p-1} & \dots & s_{2} \\
		\hdotsfor{5} \\
		s_{N-1}    		& s_{N-2} & s_{N-3} & \dots & s_{N-p} \\
	\end{bmatrix}
\end{equation}
where $s_0, \dots s_{N-1}$ are random BPSK signals chosen in advance. $\mbf{\theta}$ are the channel parameters we wish to estimate using least squares these parameters.
The mathematical setting is given by
\begin{equation}
x_n = \displaystyle \sum_{t=1}^ps_{n-t}\theta_{t}+ v_n
\end{equation}
where we define
\begin{equation}
s_{i} \doteq 0 \hsk \forall i < 0.
\end{equation}
This equation can be written as
\begin{equation}
\mbf{x} = \mbf{A}\mbf{\theta} + \mbf{v}
\end{equation}
where $\mbf{A}$ is defined in \BracketRef{eq:channel_estimation_fixed_a_definition}.
The noise vector \nl $\mbf{v} \doteq \left(v_0, \dots, v_{N-1}\right)^T$ can be modelled as:
\begin{equation}
v_n = \dispsum_{i=0}^k h_ij_{n-i} + w_n
\end{equation}
where $j_{n}$ is i.i.d zero mean bounded signal for example a BPSK signal. This can happen for example when estimating the channel in a CDMA sequence, when the interference is composed of another CDMA signal. \footnote{Similar analysis is relevant for high range resolution (HRR) estimation of target parameters in the presence of temporally correlated jammer. See ~\cite{li2000airborne} for example.} We denote by $\eta$ the bound for $j_{n}$, i.e. $P\left(j_{n} \leq \eta\right) = 1$. We also assume that $h_i$ is an unknown system. We now prove that the noise sequence is a zero mean martingale difference and that it is sub-Gaussian and thus admits assumptions A6 and A8. If so, we can use theorem \ref{thm:subgaussian_martingale_difference_fixed_matrix} to calculate the number of samples required to achieve a certain finite sample performance for this interesting model.
\begin{equation}
\begin{array}{lcl}
&E\left(v_n | F_{n-1}\right) &\\
= &E\left(\dispsum_{i = 0}^k h_ij_{n-i} + w_n  | F_{n-1}\right)& \\
= &E\left(\dispsum_{i=0}^k h_ij_{n-i} | F_{n-1}\right) + E\left(w_n | F_{n-1}\right)&\\
 = & E\left(\dispsum_{i=0}^k h_ij_{n-i} | F_{n-1}\right) & = 0.
\end{array}
\end{equation}
The second equality follows from the independence of the random variables $w_n$, $j_n$ and $h_n$. The next equality follows from the fact that $w_n$ is zero mean. The last equality follows from the fact that $E\left(j_n\right) = 0$. We now prove that $v_n$ is sub-Gaussian. We use the assumption that $j_{n} \leq \eta$ and that $h_n$ and $w_n$ are sub-Gaussian with parameter $R_1$ and $R_2$ respectively. Using these facts with the property that $j_kh_n$ is sub-Gaussian because they are independent and $j_{n} \leq \eta$ and the fact that linear combinations of sub-Gaussian random variables is sub-Gaussian ~\cite{vershynin2010introduction, eldar2012compressed} we can conclude that $v_n$ is sub-Gaussian and satisfies assumption A8. We proved that this example satisfies all the assumptions of theorem \ref{thm:subgaussian_martingale_difference_fixed_matrix} and therefore we can use the theorem to bound the number of samples needed to achieve a predefined performance. \nl
We demonstrate these results by calculating the bounds for specific parameters and running simulations. We conducted two experiments. The first aims to study the number of required samples as a function of $r$ and the second aims to study the number of required samples as a function of $\eps$. The parameters of the simulations are $p = 8$, $R = 0.1$. In the first experiment $\eps = 0.01$ and in the second experiment $r = 0.01$. In the task at hand we wish to calculate the number of samples required to ensure that the estimated channel parameters are close to the real parameters with high probability.\nl
Fig. 5 and Fig. 6 clearly demonstrate that the bounds are valid and close to the number of required samples. This example shows again the ability of this work to analyze performance of important communications problems such as channel estimation.
\section{Concluding remarks}\label{sec:concluding_remarks}
This paper examined the finite sample performance of the $L^{\infty}$ error of the linear least squares estimator. We showed very fast convergence of the number of samples required as a function of the probability of the $L^{\infty}$ error. We analyzed performance in both the fixed design matrix case and in the random design matrix case. We showed that in the random design matrix case the number of samples required to achieve a maximal deviation $r$ with probability $1-\eps$ is $N \sim O\left(\frac{1}{r^2}\log \frac{1}{\eps}\right)$ and we provided simply computable bounds in terms of $r$ and $\eps$. We demonstrated in simulations that the new bound is tighter than previous bounds in the case of i.i.d noise. Moreover, we extended the theorem to the interesting case where the noise is a bounded martingale difference sequence. This is the first finite sample analysis result for least squares under this noise model, which is highly relevant in practical applications.
We also analyzed the fixed design matrix setting. In this setting $\mbf{A}$ is not random but known in advance. We showed how we can apply our theorems to this setting as well.  We also point out that a special and interesting case for the random design matrix that is supported by our setting is when the columns are independent and $E\left(a_i^2\right) = \sigma_i$ and $E\left(a_i a_j\right) = 0$ for $j\neq i$. This will result in $\mbf{M} = \text{diag}\left\{\sigma_1, \dots, \sigma_p\right\}$. Our results are valid in this case as well.\nl
A few research problems related to this paper remain open. The first is the generalization to regularized linear least squares and the second is tightening the bounds. We also believe that the sub-Gaussian parameter can be replaced by bounds on a few moments of the distribution and can relax the bounds. Another interesting research direction is tightening the bounds, through better estimates on the tail of the distribution of eigenvalues of random matrices. These important problems are left for further study.

\section{Appendix A}\label{sec:AppendixA}
In this section we prove lemma \ref{lemma:general_least_squares_calculation}.
\begin{proof}
For the least squares problem we have
\begin{equation}
J^N_{0}\left(\mbf{\theta},\mbf{x}\right) =\left(\mbf{x} - \mbf{A}\mbf{\theta}\right)^T\left(\mbf{x} - \mbf{A}\mbf{\theta}\right).
\end{equation}
The minimum of this function is
\begin{equation}
\hat{\mbf{\theta}}^N_{0} = \left(\frac{1}{N} \left(\dispsum_{n=1}^N \mbf{a}_n\mbf{a}_n^T\right)\right)^{-1} \meansum \mbf{a}_n^Tx_n.
\end{equation}

Denoting $\mbf{C} \doteq \frac{1}{N}\left(\mbf{A^T}\mbf{A}\right) = \frac{1}{N} \left(\dispsum_{n=1}^N\mbf{a}_n\mbf{a}_n^T\right)$ we obtain

\begin{equation}\label{eq:first_elementwise_bound}
\begin{array}{lcl}
\abs{\left(\hat{\mbf{\theta}}^N_{0} - \mbf{\theta}_{0}\right)_i} &= &\left(\mbf{C}^{-1}\meansum \mbf{a}_n^T v_n\right)_i \\
&\leq &\tilde{\lambda}\left(\mbf{A}\right)\meansum a_{ni}v_n
\end{array}
\end{equation}
The last inequality is true as $\mbf{C}$ is Hermitian.
\nl
Therefore,
\begin{equation}
\begin{array}{lcl}
&P\left(\left(\hat{\mbf{\theta}^N_{0}} - \mbf{\theta}_{0}\right)_i > r \right) & \\
\leq &P\left(\abs{\tilde{\lambda}\left(\mbf{A}\right)\meansum a_{ni}v_n} > r\right) & \\
\leq& P\left(\abs{\meansum a_{ni}v_n} > \frac{r}{\tilde{\lambda}\left(\mbf{A}\right)}\right)&.
\end{array}
\end{equation}
\end{proof}

\section{Appendix B}\label{sec:AppendixB}
In this section we prove lemma \ref{lemma:omega_bound}.
\begin{proof}
We want to bound the term
\begin{equation}
\begin{array}{lcl}
& P\left(\lambda_{max}\left(\left(\frac{1}{N}\mbf{A}^T\mbf{A}\right)^{-1}\right) \geq \frac{2}{\sigma_{min}}\right) \\
= & P\left(\frac{1}{\lambda_{min}\left(\frac{1}{N}\mbf{A}^T\mbf{A}\right)} \geq \frac{2}{\sigma_{min}}\right) &\\
= & P\left(\lambda_{min}\left(\frac{1}{N}\mbf{A}^T\mbf{A}\right) \leq \frac{\sigma_{min}}{2}\right).&
\end{array}
\end{equation}
Recall that $\mbf{M} = E\left(\frac{1}{N}\mbf{A}^T\mbf{A}\right)$. We have
\begin{equation}
\begin{array}{lll}
& P\left(\lambda_{min}\frac{1}{N}\mbf{A}^T\mbf{A} \leq \frac{\sigma_{min}}{2}\right) \\
= & P\left(-\lambda_{max}\left(-\frac{1}{N}\mbf{A}^T\mbf{A}\right) \leq \frac{\sigma_{min}}{2}\right) &\\
 = &P\left(\lambda_{max}\left(-\frac{1}{N}\mbf{A}^T\mbf{A}\right) \geq -\frac{\sigma_{min}}{2}\right)  &\\
= & P\left(\lambda_{max}\left(-\frac{1}{N}\mbf{A}^T\mbf{A}\right) + \lambda_{min}\left(\mbf{M}\right) \geq \sigma_{min} - \frac{\sigma_{min}}{2}\right) &\\
=&P\left(\lambda_{max}\left(-\frac{1}{N}\mbf{A}^T\mbf{A} + \lambda_{min}\left(\mbf{M}\right)\right)\geq \frac{\sigma_{min}}{2}\right) &\\
\leq & P\left(\lambda_{max}\left(\mbf{M} - \frac{1}{N}\mbf{A}^T\mbf{A}\right) \geq \frac{\sigma_{min}}{2}\right) & \\
= & P\left(\lambda_{max}\left(\dispsum_{n=1}^N \frac{1}{N}\left(\mbf{M} - \mbf{a}_n\mbf{a}_n^T\right)\right) \geq \frac{\sigma_{min}}{2}\right),&
\end{array}
\end{equation}

where the last inequality follows from Weyl's inequality ~\cite{horn2012matrix}. Since $E\left(\frac{1}{N}\left(\mbf{a}_n\mbf{a}_n^T - \mbf{M}\right)\right) = 0$ and since $\mbf{M}$ is positive definite we have $\frac{\sigma_{min}}{2} \geq 0$ we can employ Bernstein's inequality for matrices ~\cite{Tropp}. In order to use Bernstein's inequality we need to calculate the terms $\lambda_{max}\left(\mbf{B}_n\right)$ and $\norm{E\left(\dispsum_{n=1}^N\mbf{B}_n^2\right)}$ where $\mbf{B}_n \doteq \frac{1}{N}\left(\mbf{M} - \mbf{a}_n\mbf{a}_n^T\right)$. We follow the guidelines in ~\cite{vershynin2010introduction, eldar2012compressed} to do so.

\begin{equation}
\begin{array}{lcl}
\norm{\mbf{B}_n} &\leq& \frac{1}{N}\left(\norm{\mbf{a}_n\mbf{a}_n^T} + \sigma_{max}\right)  \\ &\leq& \frac{1}{N}\left(\norm{\mbf{a}_n}_2^2 + \sigma_{max}\right) \\
&\leq& \frac{1}{N}\left(\alpha^2 p + \sigma_{max}\right) = \frac{\alpha^2 p + \sigma_{max}}{N}.
\end{array}
\end{equation}
We now compute $\norm{E\left(\mbf{B}_n^2\right)}$.
\begin{equation}
\mbf{B}_n^2 = \frac{1}{N^2}\left(\left(\mbf{a}_n\mbf{a}_n^T\right)^2 - 2\mbf{a}_n\mbf{a}_n^T\mbf{M} + \mbf{M}^2\right).
\end{equation}
Since $E\left(\mbf{a}_n\mbf{a}_n^T\right) = \mbf{M}$ we get
\begin{equation}
E\left(\mbf{B}_n^2\right) = \frac{1}{N^2}\left(E\left(\mbf{a}_n\mbf{a}_n^T\right)^2 - \mbf{M}^2\right).
\end{equation}
Calculating the norm of the first term gives
\begin{equation}
\begin{array}{lcl}
\norm{E\left(\mbf{a}_n\mbf{a}_n^T\right)^2} &=& \norm{\mbf{a}_n}_2^2\norm{E\left(\mbf{a}_n\mbf{a}_n^T\right)} \\
&\leq& p\alpha^2 \norm{E\left(\mbf{a}_n\mbf{a}_n^T\right)} = p\alpha^2\sigma_{max}.
 \end{array}
\end{equation}
Therefore,
\begin{equation}
\norm{E\left(\mbf{B}_n^2\right)} \leq \frac{p\alpha^2\sigma_{max} + \sigma_{max}^2}{N^2}.
\end{equation}
Using this we obtain
\begin{equation}
\norm{\dispsum_{n=1}^N E\left(\mbf{B}_n^2\right)} \leq \frac{p\alpha^2 \sigma_{max} + \sigma_{max}^2}{N}.
\end{equation}

Now we can use the matrix Bernstein inequality ~\cite{Tropp} to obtain
\begin{equation}
\begin{array}{lcl}
&P\left(\lambda_{max}\left(\dispsum_{n=1}^N\left(-\mbf{a}_n\mbf{a}_n^T + \mbf{M}\right)\right) \geq  \frac{\sigma_{min}}{2}\right) \\
\leq& p\exp\left(-\frac{\frac{\sigma_{min}^2}{8}}{\frac{p\alpha^2\sigma_{max} + \sigma_{max}^2}{N} + \frac{p\alpha^2 + \sigma_{max}}{3N}\frac{\sigma_{min}}{2}}\right) & \\
\leq& p\exp\left(-\frac{3}{4}\frac{N\sigma_{min}^2}{\left(6\sigma_{max} + \sigma_{min}\right)\left(p\alpha^2 + \sigma_{max}\right)}\right).&
\end{array}
\end{equation}
Choosing
\begin{equation}
N \geq \frac{4}{3}\frac{\left(6\sigma_{max} + \sigma_{min}\right)\left(p\alpha^2 + \sigma_{max}\right)}{\sigma_{min}^2}\log\left(\frac{p}{\eps'}\right) \doteq N_{rand}\left(\eps'\right)
\end{equation}
finishes the proof.
\end{proof}
\section{Appendix C}
\begin{proof}
In this section we prove theorem \ref{thm:bounded_noise_theorem}.
In order to prove the result we examine each coordinate separately once more.
In order to bound the term
\begin{equation}
P\left(\abs{\hat{\mbf{\theta}}^N_0 - \mbf{\theta}_0}_i > r\right)
\end{equation}
we use lemma \ref{lemma:general_least_squares_calculation} and analyze the term
\begin{equation}
P\left(\abs{\meansum a_{ni}v_n} > \frac{r}{\tilde{\lambda}\left(\mbf{A}\right)}\right).
\end{equation}
We start by defining the set of events
\begin{equation}
E_{rand} \doteq \left\{\mbf{x} : \tilde{\lambda}\left(\mbf{A}\right) \geq \frac{2}{\sigma_{min}}\right\}.
\end{equation}
We want to study the number of samples required to achieve that $P\left(\mbf{x} \in E_{rand}\right) \leq \frac{\eps}{3}$.
In order to do so we use lemma \ref{lemma:omega_bound} with parameter $\eps' = \frac{\eps}{3}$ to find that $\forall N > N_{rand}\left(\eps\right)$
\begin{equation}
P\left(\mbf{x} \in E_{rand}\right) = P\left(\tilde{\lambda}\left(\mbf{A}\right) \geq \frac{2}{\sigma_{min}}\right) \leq \frac{\eps}{3}.
\end{equation}
We now assume that $\mbf{x} \notin E_{rand}$.
The next part of the proof is to study the probability that given $1\leq i \leq p$, $\mbf{x}$ belongs to the set
\begin{equation}
E_2\left(i\right) = \left\{\mbf{x} : \abs{\meansum a_{ni}v_n} > \frac{r\sigma_{min}}{2}\right\}.
\end{equation}
Since $P\left(\abs{a_{ni}} \leq \alpha\right) = 1 \hsk \forall n,i$ by assumption A1 and $P\left(\abs{v_n} \leq b\right) = 1 \hsk \forall n$ by assumption A5, we have $P\left(\abs{a_{ni}v_n} \leq \alpha b\right) = 1$. Now we can use Hoeffding's inequality ~\cite{dubhashi2009concentration} to bound the large deviation probability.
\begin{equation}
\begin{array}{lcl}
&P\left(\mbf{x} \in E_2\left(i\right) \wedge \mbf{x} \notin E_{rand}\right)&\\
=&P\left(\abs{\meansum a_{ni}v_n} > \frac{r\sigma_{min}}{2}\right) = \\
=& P\left(\abs{\dispsum_{n=1}^N a_{ni}v_n} > \frac{Nr\sigma_{min}}{2}\right) &\\
\leq& 2\exp\left(-\frac{2N^2r^2\sigma_{min}^2}{4N\alpha^2b^2}\right) = 2\exp\left(-\frac{Nr^2\sigma_{min}^2}{2\alpha^2b^2}\right) &\\
\end{array}
\end{equation}
Choosing $N > \frac{2\alpha^2b^2}{r^2\sigma_{min}^2}\log\left(\frac{3p}{\eps}\right) = N_1\left(r,\eps\right)$ ensures that $\forall 1 \leq i \leq p$
\begin{multline}
P\left(\mbf{x} \in E_2\left(i\right)\wedge \mbf{x} \notin E_{rand}\right) = \\
P\left(\abs{\meansum a_{ni}v_n} > \frac{r\sigma_{min}}{2}\right) \leq \frac{2\eps}{3p}.
\end{multline}
Next, we use the union bound to bound the probability that $\mbf{x}$ is in any of the sets $E_2\left(i\right)$.
Define
\begin{multline}
E_2 \cap E_{rand}^C \doteq \\
\left\{\mbf{x} : \exists 1 \leq i \leq p \text{ such that } \mbf{x} \in E_2\left(i\right) \wedge \mbf{x} \notin E_{rand}\right\}.
\end{multline}
Using the union bound over the elements of the error vector we obtain
\begin{multline}
P\left(\mbf{x} \in E_2 \cap E_{rand}^C\right) \leq \\
\dispsum_{i=1}^p P\left(\mbf{x} \in E_2\left(i\right) \wedge \mbf{x} \notin E_{rand} \right) \leq \frac{2\eps}{3}.
\end{multline}
Using the union bound once more we achieve $\forall N > N\left(r,\eps\right)$
\begin{multline}
P\left(\mbf{x} \in \left(E_2 \cap E_{rand}^C\right) \cup E_{rand}\right) \leq \\
P\left(\mbf{x} \in E_2 \cap E_{rand}^C\right) + P\left( \mbf{x} \in \cup E_{rand}\right) \leq \frac{2\eps}{3} + \frac{\eps}{3} = \eps.
\end{multline}
This ensures that $\forall N > B\left(r,\eps\right)$
\begin{equation}
P\left(\norm{\hat{\mbf{\theta}}^N_0  - \mbf{\theta}_0}_{\infty} > r\right) \leq \eps
\end{equation}
and finishes the proof.
\end{proof}
\bibliographystyle{IEEEtran}
\bibliography{least_squares_performance}

\end{document}